\documentclass{article}[11pt]

\usepackage{graphicx,tikz,color} 
\usepackage{amsmath,amsfonts,amssymb,amsthm,latexsym}

\newtheorem{theorem}{Theorem}[section]

\newtheorem{corollary}[theorem]{Corollary}
\newtheorem{remark}[theorem]{Remark}
\newtheorem{lemma}[theorem]{Lemma}
\newtheorem{claim}{Claim}[section]

\newenvironment{claimproof}[1]{{\it\noindent{Proof.}}\space#1}{\footnotesize \hfill \ensuremath{(\square)} \medskip}

\newcommand{\gc}{\gamma_{con}}
\newcommand{\gi}{\gamma_{iso}}

\begin{document}

\title{Convex and isometric domination of (weak) dominating pair graphs}

\author{
Bo\v stjan Bre\v sar$^{a,b}$ \and
Tanja Gologranc$^{a,b}$ \and Tim Kos$^{b}$ }

\maketitle

\begin{center}
$^a$ Faculty of Natural Sciences and Mathematics, University of Maribor, Slovenia\\
\medskip

$^b$ Institute of Mathematics, Physics and Mechanics, Ljubljana, Slovenia\\
\medskip

\end{center}

\begin{abstract}
A set $D$ of vertices in a graph $G$ is a dominating set if every vertex of $G$, which is not in $D$, has a neighbor in $D$. A set of vertices $D$ in $G$ is convex (respectively, isometric), if all vertices in all shortest paths (respectively, all vertices in one of the shortest paths) between any two vertices in $D$ lie in $D$. The problem of finding a minimum convex dominating (respectively, isometric dominating) set is considered in this paper from algorithmic point of view. For the class of weak dominating pair graphs (i.e.,~the graphs that contain a dominating pair, which is a pair of vertices $x,y\in V(G)$ such that vertices of any path between $x$ and $y$ form a dominating set), we present an efficient algorithm that finds a minimum isometric dominating set of such a graph. On the other hand, we prove that even if one restricts to weak dominating pair graphs that are also chordal graphs, the problem of deciding whether there exists a convex dominating set bounded by a given arbitrary positive integer is NP-complete. By further restricting the class of graphs to chordal dominating pair graphs (i.e.,~the chordal graphs in which every connected induced subgraph has a dominating pair) we are able to find a polynomial time algorithm that determines the minimum size of a convex dominating set of such a graph.
\end{abstract}


\noindent
{\bf AMS Subj.\ Class.\ (2010)}: 05C85, 05C69, 05C12, 68E10

\section{Introduction}

Domination theory is one of the classical and most studied topics of graph theory; it was surveyed in two monographs that were published almost twenty years ago~\cite{hhs,hhs1}, and the theory has been extensively developed also in the last two decades. While in the basic version of domination, a {\em dominating set} $D$ is a set of vertices in a graph $G$ such that any vertex of $G$ not in $D$ has some neighbor in $D$, many variations of this concept have been introduced. In particular, in the so-called  {\em connected domination}, as introduced in~\cite{SW1979}, a dominating set $D$ is required to induce a connected subgraph. The idea reflects the requirements of potential applications, where vertices in $D$ represent locations/nodes of discrete network, in which monitoring devices are placed that monitor the nodes in their closed neighborhoods, and it is desirable that one can move between locations/nodes, which are in $D$, by passing only through location/nodes that are in $D$. In the more restrictive case in which the time of moving between different nodes in $D$ is also important, one can require that some shortest path between any two nodes in $D$ lies completely in $D$ (representing the so-called {\em weakly convex} or {\em isometric domination}); or, even more restrictively, that any shortest path between any two nodes in $D$ lies completely in $D$ (which then yields the so-called {\em convex domination}).

Two graph invariants appear in this context. The {\em convex domination number} of a graph $G$, $\gc(G)$, is the minimum cardinality of a set $D\subseteq V(G)$ such that $D$ is at the same time a dominating set and a convex set (recall that a set $D$ is convex if for any two vertices $x,y\in D$ all shortest $x,y$-paths lie in $D$). The {\em isometric domination number} of a graph $G$, $\gi(G)$, is the minimum cardinality of a set $D\subseteq V(G)$ such that $D$ is at the same time a dominating set and an isometric set, where the latter means that for any two vertices $x,y\in D$ there exists a shortest $x,y$-path that lies in $D$.  The study of convex domination and of isometric domination (introduced under then name  weak convex domination) was initiated in 2004 by Lema\'{n}ska~\cite{L2004} and Raczek~\cite{R2004}, and was further studied from different points of view in several papers~\cite{LC2012,L2010, LRG2012,RL2010}. Raczek proved that the decision versions of isometric and convex domination number of a graph are NP-complete, even for bipartite and split graphs~\cite{R2004} (and hence also for chordal graphs). In fact, determining these numbers in split graphs is easily seen to be equivalent to the {\sc Set Cover Problem}, one of the fundamental NP-complete problems due to Karp~\cite{karp}, see also~\cite{GJ1979}.

The algorithmic and complexity issues were investigated recently for several other convexity parameters~\cite{BCD,CDS,DPR}. The theory of convex sets in graphs and other discrete structures was surveyed in the monograph already in 1993~\cite{vvel-93}, and it encompasses several important results in metric graph theory. In this developed part of graph theory (see also a survey on metric graph theory and geometry~\cite{BaCh}) it is common to use the word {\em isometric subgraph} for a distance-preserving subgraph, while weak convexity usually refers to some form of convexity related to vertices of small distance. From this reason we suggest the name {\em isometric domination} instead of {\em weak domination}.

It is natural to consider these concepts in classes of graphs in which one can easily find nontrivial dominating sets, which are at the same time convex or isometric sets (nontrivial in this case means that the sets are not equal to $V(G)$). In particular, it is easy to see that removing all simplicial vertices in a chordal graph $G$, yields a subset of $V(G)$, which is convex and dominating. As mentioned above, the exact convex domination number is hard in split graphs and hence also in chordal graphs. Another interesting class of graphs in this context is that of asteroidal-triple-free graphs (AT-free graphs, for short); these graphs are defined as the graphs containing no {\em{asteroidal triples}}, i.e.\ independent sets of three vertices such that each pair is joined by a path that avoids the neighborhood of the third. The class of AT-free graphs contains many known classes of graphs such as interval, permutation, trapezoid, and co-comparability graphs, which have interesting geometric representations, and have also been in the focus of algorithmic graph theory, e.g.~see the monographs~\cite{bls-99,mcmc-99}. In~\cite{COS1997} Corneil, Olariu and Stewart presented the evidence that the absence of asteroidal triples imposes linearity of the recognition of the mentioned four classes. They also proved that AT-free graphs contain a {\em dominating pair}, that is, a pair of vertices with the property that every path connecting them is a dominating set. A linear time algorithm to find a dominating pair in AT-free graphs was presented in~\cite{COS2005}.

More generally, a graph is called a {\em weak dominating pair graph} if it contains a dominating pair, while a graph is called a {\em dominating pair graph} if each of its connected induced subgraphs is a weak dominating pair graph. Both graph classes contain AT-free graphs, and were introduced by Deogun and Kratsch in~\cite{DK2002}, where also a characterization of chordal dominating pair graphs using forbidden induced subgraphs was proven. In~\cite{PCK2004} it was shown that chordal dominating pair graphs can be recognized in polynomial time.

In this paper, we prove that convex domination problem is NP-complete when restricted to chordal weak dominating pair graphs (see Section~\ref{sec:weakDom}). On the other hand, we present in Section~\ref{sec:pairDom} a polynomial time algorithm to determine the convex domination number of an arbitrary chordal dominating pair graph. (As a corollary, the convex domination number of an interval graph can also be computed in polynomial time.)
Finally, in Section~\ref{sec:iso} we give a polynomial time algorithm to determine the isometric domination number of a (weak) dominating pair graph in which a dominating pair is also given. (Since one can determine a dominating pair in AT-free graphs in polynomial time, the problem of isometric domination number is polynomial in these graphs.) We conclude the introduction by remarking that results in these paper demonstrate that complexity behaviour of convex and isometric domination problems can be significantly different; see Figure~\ref{fig:classes} presenting the classes of graphs considered in this paper.

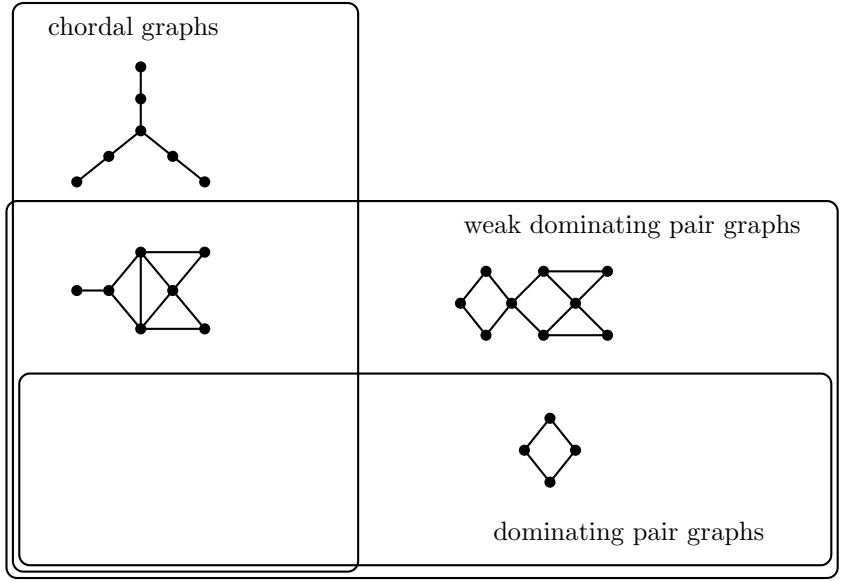
\begin{figure}[!ht]
	
	\centering
	\begin{tikzpicture}[scale=0.85, style=thick]
	\def\vr{3pt}
	\def\len{1}
	
	\draw[rounded corners] (0,0.5) rectangle (13,6.4); 
	\draw[rounded corners] (0.1,0.6) rectangle (5.5,9.5); 
	\draw[rounded corners] (0.2,0.7) rectangle (12.9,3.7); 
	
	\draw (7,6) node[right] {weak dominating pair graphs};
	\draw (12, 1.2) node[left] {dominating pair graphs};
	\draw (.5,9.1) node[right] {chordal graphs};

	
	
	\filldraw [fill=black, draw=black,thick] (1.1,6.7) circle (2pt);
	\filldraw [fill=black, draw=black,thick] (1.6,7.1) circle (2pt);
	\filldraw [fill=black, draw=black,thick] (2.1,7.5) circle (2pt);
	\filldraw [fill=black, draw=black,thick] (2.1,8) circle (2pt);
	\filldraw [fill=black, draw=black,thick] (2.1,8.5) circle (2pt);
	\filldraw [fill=black, draw=black,thick] (2.6,7.1) circle (2pt);
	\filldraw [fill=black, draw=black,thick] (3.1,6.7) circle (2pt);
	
	\draw (1.1,6.7) -- (2.1,7.5);
	\draw (2.1,7.5) -- (2.1,8.5);
	\draw (2.1,7.5) -- (3.1,6.7);
	
	
	\filldraw [fill=black, draw=black,thick] (1.1,5) circle (2pt);
	\filldraw [fill=black, draw=black,thick] (1.6,5) circle (2pt);
	\filldraw [fill=black, draw=black,thick] (2.1,5.6) circle (2pt);
	\filldraw [fill=black, draw=black,thick] (2.1,4.4) circle (2pt);
	\filldraw [fill=black, draw=black,thick] (2.6,5) circle (2pt);
	\filldraw [fill=black, draw=black,thick] (3.1,5.6) circle (2pt);
	\filldraw [fill=black, draw=black,thick] (3.1,4.4) circle (2pt);
	
	\draw (1.1,5) -- (1.6,5);
	\draw (2.1,5.6) -- (1.6,5);
	\draw (2.1,4.4) -- (1.6,5);
	\draw (2.1,5.6) -- (2.1,4.4);
	\draw (2.1,4.4) -- (2.6,5);
	\draw (2.1,5.6) -- (2.6,5);
	\draw (2.1,4.4) -- (3.1,4.4);
	\draw (2.1,5.6) -- (3.1,5.6);
	\draw (3.1,4.4) -- (2.6,5);
	\draw (3.1,5.6) -- (2.6,5);
	
	
	\filldraw [fill=black, draw=black,thick] (7.1,4.8) circle (2pt);
	\filldraw [fill=black, draw=black,thick] (7.5,5.3) circle (2pt);
	\filldraw [fill=black, draw=black,thick] (7.5,4.3) circle (2pt);
	\filldraw [fill=black, draw=black,thick] (7.9,4.8) circle (2pt);
    
	\filldraw [fill=black, draw=black,thick] (8.4,5.3) circle (2pt);
	\filldraw [fill=black, draw=black,thick] (8.4,4.3) circle (2pt);
	\filldraw [fill=black, draw=black,thick] (8.9,4.8) circle (2pt);
	\filldraw [fill=black, draw=black,thick] (9.4,5.3) circle (2pt);
	\filldraw [fill=black, draw=black,thick] (9.4,4.3) circle (2pt);
	
	\draw (7.1,4.8) -- (7.5,5.3);
	\draw (7.1,4.8) -- (7.5,4.3);
	\draw (7.9,4.8) -- (7.5,5.3);
	\draw (7.9,4.8) -- (7.5,4.3);
	\draw (7.9,4.8) -- (8.4,5.3);
	\draw (7.9,4.8) -- (8.4,4.3);
	\draw (8.9,4.8) -- (8.4,5.3);
	\draw (8.9,4.8) -- (8.4,4.3);
	\draw (9.4,5.3) -- (8.4,5.3);
	\draw (9.4,4.3) -- (8.4,4.3);
	\draw (8.9,4.8) -- (9.4,4.3);
	\draw (8.9,4.8) -- (9.4,5.3);
	

	\filldraw [fill=black, draw=black,thick] (8.1,2.5) circle (2pt);
	\filldraw [fill=black, draw=black,thick] (8.5,3) circle (2pt);
	\filldraw [fill=black, draw=black,thick] (8.5,2) circle (2pt);
	\filldraw [fill=black, draw=black,thick] (8.9,2.5) circle (2pt);

	\draw (8.1,2.5) -- (8.5,3);
	\draw (8.1,2.5) -- (8.5,2);
	\draw (8.9,2.5) -- (8.5,3);
	\draw (8.9,2.5) -- (8.5,2);

	\end{tikzpicture}
	
	\caption{Classes of graphs studied in the paper; examples showing proper inclusions.}
	\label{fig:classes}
\end{figure}

\section{Preliminaries} \label{sec:prelim}
All graphs considered in this paper are finite, simple, and undirected.
The {\em neighborhood} of a vertex $v\in V(G)$ is the set $N_G(v)=\{u\in V(G)\,:\,uv\in E(G)\}$, while {\em neighborhood of a set} $X\subseteq V(G)$ is defined as $N_G(X)= \bigcup_{v\in X}N_G(v)$. The {\em closed neighborhood} of a vertex $v\in V(G)$ is the set $N_G[v]=N(v) \cup \{v\}$, while {\em closed neighborhood of a set} $X\subseteq V(G)$ is defined as $N_G[X]= \bigcup_{v\in X}N_G[v]$. 
Given a set $X\subseteq V(G)$ and a vertex $u\in X$, we define $pn_G(u,X)$ as the set $\{w\in V(G)\,:\, N_G[w]\cap X = \{u\}\}$. A member of the set $pn_G(u,X)$ is said to be an {\em $X$-private neighbor of $u$ in $G$}.

Let $X \subseteq V(G)$ be any subset of vertices of $G$. The subgraph of $G$ induced by vertices of $X$ will be denoted by $\left\langle  X \right\rangle$.
A {\em clique} of a graph $G$ is a set $C \subseteq V(G)$ such that $\left\langle  C \right\rangle$ is a complete graph. An {\em independent set} of a graph $G$ is a set $I \subseteq V(G)$, no two vertices of which are adjacent.

A {\em dominating set} of a graph $G$ is a set $D \subseteq V(G)$ such that every vertex not in $D$ is adjacent to at least one vertex from $D$.
If $X$ and $Y$ are subsets of vertices in $G$, then $X$ {\em dominates} $Y$ in $G$ if $Y \subseteq N_G[X]$.

A set $S \subseteq V(G)$ is a {\em convex set}, if for any two vertices $u,v \in S$ the set $S$ contains all the vertices that lie on a shortest path between $u$ and $v$. Given a set $T\subseteq V(G)$, the {\em convex hull} of $T$, denoted $CH(T)$, is the smallest convex set that contains $T$. It is obvious that $S=CH(S)$ if and only if $S$ is a convex set. It is also easy to see that $T\subseteq S$ implies $CH(T)\subseteq CH(S)$. 

A set $S\subseteq V(G)$ is called {\em isometric}, if for any two vertices $u,v\in S$ there exists a shortest $u,v$-path whose all vertices are in $S$. By $d_G(u,v)$ we denote the distance between vertices $u$ and $v$, which is defined as the length of a shortest $u,v$-path in a graph $G$. The {\em diameter} ${\rm diam}(G)$ of a graph $G$ is defined as $\max_{u,v\in V(G)}\{d_G(u,v)\}$.  Using this notation, a subset $S$ is isometric if and only if $d_G(u,v)=d_H(u,v)$ for any two vertices $u,v\in S$, where $H$ is the subgraph of $G$ induced by vertices in $S$.

A {\em convex dominating set}, abbreviated a CD-set, of a graph $G$, is a set of vertices that is convex and dominating.
The {\em convex domination number} of $G$, denoted by $\gc(G)$, is the minimum cardinality of a CD-set of $G$. A CD-set of $G$ of cardinality $\gc(G)$ will be referred to as a $\gc$-set of $G$. A CD-set $D$ is {\em minimal CD-set} of a graph $G$, if no proper subset of $D$ is a CD-set. Similarly, an {\em isometric dominating set}, abbreviated an ID-set, of a graph $G$, is a set of vertices that is isometric and dominating. The {\em isometric domination number} of $G$, denoted by $\gi(G)$, is the minimum cardinality of an ID-set of $G$. An ID-set of $G$ of cardinality $\gi(G)$ will be referred to as a $\gi$-set of $G$.

A {\em split graph} is a graph whose vertex set can be partitioned into a clique and an independent set. Split graphs are contained in the class of chordal graphs, i.e.\ graphs with no induced cycles of length more than 3. A {\em simplicial vertex} is a vertex whose neighborhood is a clique. Intersection graphs of the intervals on the real line are called {\em interval graphs}, see~\cite{bls-99} for basic properties, and~\cite{acfgm-07,wg-1986} for applications. In our context, it is interesting to note that a graph is an interval graph if and only if it is chordal and asteroidal triple-free graph~\cite{LB1962}.

A pair $(x, y)$ of vertices of a graph $G$ is a {\em dominating pair} if, for every path $P$ between $x$ and $y$, the vertex set $V(P)$ is a dominating set of $G$. (It is worth mentioning that $x = y$ is allowed.)
A graph $G$ is a {\em weak dominating pair graph} if $G$ has a dominating pair.
A graph $G$ is a {\em chordal weak dominating pair graph} if $G$ is a chordal graph and weak dominating pair graph.
A graph $G$ is a {\em dominating pair graph} if every connected induced subgraph of $G$ has a dominating pair. A graph $G$ is a {\em chordal dominating pair graph} if $G$ is a chordal graph and dominating pair graph.

\section{Convex domination of chordal weak dominating pair graphs}
\label{sec:weakDom}
In this section we show that {\sc Convex Dominating Set Problem} on chordal weak dominating pair graphs is NP-complete.

\begin{center}
\fbox{\parbox{0.99\linewidth}{\noindent
{\sc Convex Dominating Set Problem}\\[.8ex]
\begin{tabular*}{.95\textwidth}{rl}
{\em Input:} & A connected graph $G = (V,E)$ and a positive integer $k$.  \\
{\em Question:} & Does $G$ have a convex dominating set of size $\leq k$? \\
\end{tabular*}
}}
\end{center}

\begin{lemma}
	\label{lem:split}
	Let $G = (V(G),E(G))$ be a connected split graph with a maximum clique $C$ and an independent set $I = V(G)\setminus C$. If $D$ is a minimal CD-set of $G$, then $D \subseteq C$.
\end{lemma}
\begin{proof}
Let $D$ be a minimal CD-set of a split graph $G$ with a maximum clique $C$ and an independent set $I = V(G)\setminus C$. 
	Suppose that $D \cap I \neq \emptyset$ and let $v \in D \cap I$. If $|D| = 1$, then $N[v] = V(G)$. This contradicts maximality of $C$, since $v \notin C$. Hence $|D| \geq 2$. Since $D$ is a convex set, there exists $w \in D \cap N(v)$. Note that $w \in C$ and that $N[v] \subseteq N[w]$, hence $v$ has no $D$-private neighbor. Vertex $v$ also does not lie on any shortest path between two vertices from $D-v$, which implies that $D-v$ is a CD-set, contradicting minimality of $D$.
\end{proof}

\begin{theorem}
	{\sc Convex Dominating Set Problem} is NP-complete for chordal weak dominating pair graphs.
\end{theorem}
\begin{proof}
It is easy to see that {\sc Convex Dominating Set Problem} is in NP. Indeed, one can check in linear time that a given set $D$ of $k$ vertices from a graph $G$ is dominating;  using shortest path algorithms one can also check in polynomial time, whether  $D$ is convex.
To prove that {\sc Convex Dominating Set Problem} is NP-complete for chordal weak dominating pair graphs we use a polynomial reduction from {\sc Convex Dominating Set Problem} for split graphs, which is known to be NP-complete~\cite[Theorem 3]{R2004}.
	 
	 Let $G' = (V(G'),E(G'))$ be an arbitrary connected split graph with a maximum clique $C$ and an independent set $I = V(G')\setminus C$.
Let $G = (V(G), E(G))$ be the graph defined as follows:
	 $$V(G) = V(G') \cup \{x,y,y'\} \text{ and}$$
	 $$E(G) = E(G') \cup \{(x, g')\, |\, g' \in V(G') \}  \cup \{(y, c)\, |\, c \in C \}  \cup \{(y, y')\}.$$
	 First we show that $G$ is a chordal weak dominating pair graph with dominating pair $(x,y)$. It is easy to see that $G$ is a chordal graph. Indeed, as $G'$ is a split graph, the graph obtained from $G'$ by adding vertex $x$ is still a split graph, and, when $y$ is added next, $y$ is a simplicial vertex, and therefore the resulting graph remains chordal. Finally, by adding $y'$, which is also a simplicial vertex, we get that $G$ is a chordal graph. The set $\{x,y\}$ already dominates the whole graph, therefore every $x,y$-path dominates $G$. It follows that $G$ is a chordal weak dominating pair graph with $(x,y)$ as dominating pair. It is also clear that $G$ can be constructed from $G'$ in polynomial time.
	 
\begin{claim} If $k$ is an integer, $k\ge 1$, then $G'$ has a convex dominating set of size at most $k$, if and only if $G$ has a convex dominating set of size at most $k+1$.
\end{claim}
\begin{claimproof} Let $D'$ be a minimal CD-set of $G'$ with $|D'| = k$. By Lemma \ref{lem:split}, $D' \subseteq C$. We claim that $D = D' \cup \{y\}$ is a CD-set of $G$. Since $D'\subseteq C \subseteq N(y)$, $D$ is a clique, and $D$ is a convex set in $G$. Vertices $x$ and $y$ are dominated by $D'$, while vertex $y'$ is dominated by $y$. Hence $D$ is a CD-set of $G$ with cardinality $k+1$.
	 
	 Let $D$ be a minimal CD-set of $G$ with $|D| = k+1$. Firstly, we show that $y' \notin D$. Suppose that $y' \in D$. Since $D$ is a dominating set, at least one vertex of $N[x]$ has to be in $D$. Vertex $y$ lies on all shortest paths between $y'$ and vertices in $N[x]$, therefore $y \in D$. Hence $D\setminus\{y'\}$ is a CD-set, contradicting minimality of $D$. This also implies that $y \in D$.
	 
	 Next, we  show that $x \notin D$. Suppose that $x \in D$. Since $D$ is convex and we already know that $y \in D$, all vertices of $C$ have to be in $D$. We infer that $C \cup \{y\}\subsetneq D$, and $C\cup\{y\}$ is a CD-set of $G$, which again contradicts the minimality of $D$.
	 
	 Finally, we show that $D \cap I = \emptyset$. Suppose that $v \in D \cap I$. Again, since $D$ is convex and $y \in D$, there exists $w \in D \cap N(v)$. Now, we use the same arguments as in the proof of Lemma \ref{lem:split} to show that $D \cap I = \emptyset$. Hence, $D \subseteq C \cup \{y\}$.
	 
	  We claim that $D' = D\setminus\{y\}$ is a CD-set of $G'$. Since $D'$ is a clique, it is a convex set in $G'$, and because $pn_G(y,D) = \{y'\}$, we infer that $D'$ is also a dominating set in $G'$. Finally, this implies that $D'$ is a CD-set of $G'$ with cardinality $k$.
\end{claimproof}

By the above claim, the existence of a polynomial time algorithm for determining whether $\gc(G)\le k+1$, where $G$ is an arbitrary chordal  weak dominating pair graph, implies the existence of a polynomial time algorithm for determining whether $\gc(G)\le k$, where $G$ is an arbitrary split graph. By the NP-completeness of the latter problem, we derive that {\sc Convex Dominating Set Problem} is NP-complete for chordal weak dominating pair graphs.
\end{proof}

\section{Convex domination of chordal dominating pair graphs}\label{sec:pairDom}

In this section we will prove that a convex dominating set of a chordal dominating pair graphs can be found in polynomial time.  We will be using the following characterization of chordal dominating pair graphs.

\begin{theorem}{\cite[Theorem 5.3]{DK2002}}
\label{thm:forbidden}
	A chordal graph $G$ is a dominating pair graph if and only if it
	does not contain the graphs $A_1$ and $B_n$ $(n \geq 1)$ as an induced subgraph (see Figure~\ref{fig:Forbidden}).
\end{theorem}

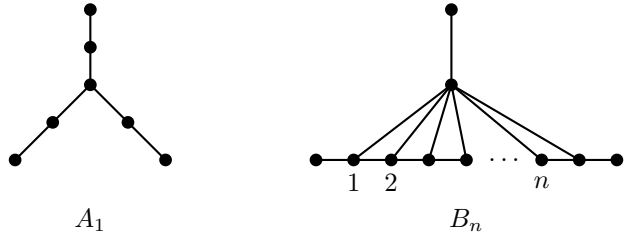
\begin{figure}[ht!]
	\begin{center}
		\begin{tikzpicture}[scale=1.0,style=thick,x=1cm,y=1cm]
		\def\vr{2.5pt} 
		
		
		\filldraw [fill=black, draw=black,thick] (0,0) circle (2pt);
		\filldraw [fill=black, draw=black,thick] (0.5,0.5) circle (2pt);
		\filldraw [fill=black, draw=black,thick] (1.5,0.5) circle (2pt);
		\filldraw [fill=black, draw=black,thick] (2,0) circle (2pt);

		\foreach \y  in {1,1.5,2}
		{
			\filldraw [fill=black, draw=black,thick] (1,\y) circle (2pt);
		}

		\foreach \x  in {4,4.5,5,5.5,6,7,7.5,8}
		{
			\filldraw [fill=black, draw=black,thick] (\x,0) circle (2pt);
		}
		\filldraw [fill=black, draw=black,thick] (5.8,1) circle (2pt);
		\filldraw [fill=black, draw=black,thick] (5.8,2) circle (2pt);
		
		\draw (4,0) -- (6,0);
		\draw (7,0) -- (8,0);
		\draw (5.8,1) -- (5.8,2);
		
		\foreach \x  in {4.5,5,5.5,6,7,7.5}
		{
			\draw (\x,0) -- (5.8,1);
		}
		
		\draw (0,0) -- (0.5,0.5);
		\draw (0.5,0.5) -- (1,1);
		\draw (1,1) -- (1.5,0.5);
		\draw (1.5,0.5) -- (2,0);
		\draw (1,1) -- (1,1.5);
		\draw (1,1.5) -- (1,2);
		
		\draw (1,-0.8) node {{\bf $A_1$}}; 
		\draw (6,-0.8) node {{\bf $B_n$}}; 
		\draw (4.5,-0.3) node {$1$}; 
		\draw (5,-0.3) node {$2$}; 
		\draw (7,-0.3) node {$n$}; 
		\draw (6.5,0) node {$\ldots$}; 
		
		\end{tikzpicture}
	\end{center}
	\caption{Forbidden graphs in Theorem~\ref{thm:forbidden}.}
	\label{fig:Forbidden}
\end{figure}

We follow with two easy lemmas about chordal dominating pair graphs with a given dominating pair (the first one follows directly from the fact that a convex dominating set is a dominating set).

\begin{lemma}\label{l:neighbours}
	Let $G$ be a chordal dominating pair graph and let $x,y$ be an arbitrary dominating pair in $G$. Then any minimum convex dominating set of $G$ contains at least one vertex from $N[x]$ and at least one vertex from $N[y]$.
\end{lemma}

\begin{lemma}\label{l:hullxy}
	Let $G$ be a chordal dominating pair graph and let $x,y$ be an arbitrary dominating pair in $G$. Then $\gc(G) \leq |CH(\{x,y\})|$.
\end{lemma}
\begin{proof}
	It suffices to prove that $S=CH(\{x,y\})$ is a convex dominating set in $G$. Since $S$ is a convex hull of $\{x,y\}$, it is clearly convex, thus $S$ contains at least one $x,y$-path $P$. As $(x,y)$ is a dominating pair, $P \subseteq S$ is a dominating set. 
\end{proof}

\begin{remark}
\label{rem:uni}
	Let $G$ be an arbitrary graph.  A graph $G$ has a universal vertex if and only if $\gc(G)=1$.
\end{remark}

\begin{remark}
	\label{rem:cg4}
	Let $G$ be an arbitrary graph. If $\gc(G) \leq 4$, then there exists a set $R\subseteq V(G)$ with $|R|\le 4$ such that $CH(R)$ is a CD-set of size $\gc(G)$.
\end{remark}

The following (our main) result shows that a smallest convex dominating set of a chordal dominating pair graph $G$ is realizable as the convex hull of some set on at most four vertices in $G$. Since there are polynomially many such sets in any graph, and the convex hull of any set of vertices in a graph can also be computed in polynomial time, we derive that convex domination number of a chordal dominating pair graph can be computed in polynomial time.

\begin{theorem}
	\label{thm:main}
Let $G$ be a chordal dominating pair graph. Then there exists a set $R\subseteq V(G)$ with $|R|\le 4$ such that $CH(R)$ is a CD-set of size $\gc(G)$.
\end{theorem}
\begin{proof}
Let $G$ be a chordal dominating pair graph, and let $(x,y)$ be its dominating pair. 
In the proof we will consider different cases with respect to the distance between $x$ and $y$, and the appearance of $x$ and $y$ in a convex dominating set $S$ of $G$, and in each of the cases we will establish the existence of a set $R$ with $|R|\le 4$ such that $CH(R)\subseteq S$, and $CH(R)$ is a CD-set of $G$. By the reasoning preceding the theorem, this implies that one can find a minimum convex dominating set in a chordal dominating pair graphs in polynomial time. 

If $G$ has a universal vertex or $\gc(G) \leq 4$, then by Remark~\ref{rem:uni} and \ref{rem:cg4} the assertion of the theorem is clear. Hence, we assume in the rest of the proof that $G$ is a chordal dominating pair graph with $\Delta(G) \leq |V(G)|-2$ and $\gc(G) > 4$. Let $(x,y)$ be an arbitrary dominating pair in $G$. In addition, we may  assume that $x$ and $y$ are not adjacent, because if $xy\in E(G)$, then $\{x,y\}$ is a CD-set, which is clearly minimum, and $\gc(G)=2$. (Note that if $R$ is a single vertex or two adjacent vertices, then its convex hull coincides with $R$.)

In the proof of Lemma~\ref{l:hullxy} we established that $CH(\{x,y\})$ is a convex dominating set of $G$. Hence, if $CH(\{x,y\})$ is a CD-set of size $\gc(G)$, then $R=\{x,y\}$. Note that if $S$ is a set of vertices such that $\{x,y\}\subseteq S$, then $CH(\{x,y\})\subseteq CH(S)$. Hence, if $CH(\{x,y\})$ is not a minimum CD-set (of size $\gc(G)$), then a minimum CD-set does not contain both $x$ and $y$. 
From this reason we may restrict our attention to CD-sets $S$ with $\{x,y\}\not\subseteq S$.

\smallskip

Recall that $d(x,y)>1$, and let us first assume that $d(x,y) \geq 3.$ 

\noindent {\bf Case 1.} $d(x,y) \geq 3.$

\smallskip

Let $S$ be a CD-set of $G$ such that $\{x,y\}\not\subseteq S$.  Lemma~\ref{l:neighbours} implies that $S$ contains at least one vertex from $N[x]$ and at least one vertex from $N[y]$. We distinguish three cases.

\begin{enumerate}
	\item For $u \in N(x)$, $S$ is a convex dominating set containing $\{u,y\}$ (and not containing $x$). 
	\item For $z \in N(y)$, $S$ is a convex dominating set containing $\{x,z\}$ (and not containing $y$).
	\item For $u \in N(x), z \in N(y)$, $S$ is a convex dominating set containing $\{u,z\}$ (and not containing $x$ nor $y$). 
\end{enumerate}

To conclude the proof of the theorem for graphs where $d(x,y) \geq 3$ it suffices to prove that in each of the above cases there exists a set of vertices $R$ with $|R|\le 4$ such that $CH(R)\subseteq S$ and $CH(R)$ is a CD-set of $G$.

\smallskip

\noindent {\bf Case 1.1} $S$ is a CD-set containing $\{u,y\}$, where $u \in N(x)$, and $x\notin S$.

\smallskip

Let $D=CH(\{u,y\})$. Clearly, $D\subseteq S$. Note that $d(u,y) \geq 2$, since $d(x,y)\geq 3$.

\begin{claim}\label{l:AnotDomin}
	Vertices in $V(G)\setminus N(x)$ are dominated by $D$.
\end{claim}
\begin{claimproof}
	Let $P$ be any shortest $u,y$-path of $G$. Since $D$ is a convex set containing $\{u,y\}$, $P$ is contained in $D$. As $(x,y)$ is dominating pair, $\{x\}\cup V(P)$ consists of the vertices of an $x,y$-path and is thus a dominating set of $G$. Therefore $V(P)$ dominates all vertices of $G$ except perhaps some vertices from $N(x)$. 
\end{claimproof}

Let $A$ be the set of vertices in $N(x)$ that are not dominated by $D$ and let $B\subseteq N(x)$ be the set of vertices not in $D$ but dominated by $D$, i.e.,\ $B=N(x)\cap (N[D]-D)$. Furthermore let $L=V(G)-N[x]-N[y]-D$.

If $A=\emptyset$, then $D=CH(\{u,y\})$ itself is a convex dominating set $S$ containing $\{u,y\}$, and we may take $R$ to be $\{u,y\}$ (where following the above notation $R$ is a set with at most 4 vertices such that $CH(R)\subseteq S$).

 Suppose now that $A$ is not empty. We will establish some properties of sets $A$ and $B$.

\begin{claim}\label{l:Aisolated}
	Let $a$ be an arbitrary vertex from $A$. If $l \in L \cup D \cup N[y]$, then $al \notin E(G)$.
\end{claim}
\begin{claimproof}
	As $A$ is a set of vertices not dominated by $D$, $a$ has no neighbors in $D$. Since $d(x,y) \geq 3$, $ay \notin E(G)$. Finally, let $l \in L \cup N(y)$.
	Suppose that $al\in E(G)$. Let $d$ be a neighbor of $l$ in $D$, and let $P$ be a shortest $d,u$-path in $G$ (note that $d=u$ is also possible). Since $D$ is convex, $P \subseteq D$. Let $u'$ be the last neighbor of $l$ on $P$ and let $Q$ be the $u',u$-subpath of $P$. Since $au\notin E(G)$ and $xl\notin E(G)$, we derive that $x,a,l,Q,x$ is an induced cycle of length at least 4, a contradiction with $G$ being chordal. 
\end{claimproof}

\begin{claim}\label{l:aNotInD}
	If $a\in A$, then $a\notin S$. 
\end{claim} 
\begin{claimproof}
	Since $a$ is not dominated by $D$, $a$ is not adjacent to $u$. Therefore any convex set that contains $D \cup \{a\}$, also contains $x$, a contradiction with $x\notin S$.
\end{claimproof}

Since $A \neq \emptyset$ and $S$ is a dominating set, we have $N[A]\cap S\neq\emptyset$. From Claims~\ref{l:Aisolated} and~\ref{l:aNotInD} (and since $x\notin S$), we derive that $N[A]\cap S\subseteq B$. Let $B_1=N[A]\cap S$. (Clearly, $B_1\subseteq B$.) 

\begin{claim}\label{l:Bclique}
The set $B_1 \cup (D\cap N(x))$ is a clique.
\end{claim}
\begin{claimproof}
	Suppose that $b_1$ and $b_2$ are two nonadjacent vertices from $B_1 \cup (D\cap N(x))$. Then $b_1,x,b_2$ is a shortest $b_1,b_2$-path which implies that $x \in CH(D \cup B_1)(\subseteq S$), a contradiction.
\end{claimproof}

\begin{claim}\label{l:linearOreder}
The sets from the family $\{N_{<A>}(b_i)\,:\, b_i \in B_1\}$ are linearly ordered with respect to inclusion.
\end{claim}
\begin{claimproof}
	Suppose that there exist $b_1,b_2 \in B_1$ such that $N_{<A>}(b_1) \nsubseteq N_{<A>}(b_2)$ and $N_{<A>}(b_2) \nsubseteq N_{<A>}(b_1)$. Therefore there exist $a_1,a_2\in A$ such that $a_1b_1,$ $a_2b_2 \in E(G)$, and $a_2b_1,a_1b_2 \notin E(G)$. Note that $a_1a_2 \notin E(G)$, as $G$ is chordal. It follows from Claim~\ref{l:Bclique} that $b_1u, b_2u,b_1b_2 \in E(G)$. Furthermore, Claim~\ref{l:Aisolated} implies that $a_1u,a_2u\notin E(G)$. Let $P=u,u_1,\ldots , u_k,y$ be a shortest $u,y$-path in $G$. As $D$ is convex, $P \subseteq D$. Note that $u_2$ is not adjacent to neither of $b_1,b_2$. Indeed, if $u_2b_1 \in E(G)$ ($u_2b_2\in E(G)$), then $b_1$ ($b_2$) lies on a shortest path between two vertices $u,u_2$ from $D$, which implies that $b_1$ ($b_2$) is in $D$, a contradiction. 
	If $u_1b_1,u_1b_2 \notin E(G)$, then vertices $a_1,a_2,b_1,b_2,u,u_1$ induce a graph $B_1$ from Figure~\ref{fig:Forbidden}, which implies that $G$ is not a chordal dominating pair graph, a contradiction. If $u_1$ is adjacent to one vertex from $\{b_1,b_2\}$, say $b_1$, then $a_1,a_2,b_1,b_2,u,u_1,u_2$ induce  a graph $B_2$, a contradiction. Finally, if $u_1$ is adjacent to both vertices $b_1$ and $b_2$, then $a_1,a_2,b_1,b_2,u_1,u_2$ induce a graph $B_1$, a contradiction.
\end{claimproof}

From Claim~\ref{l:linearOreder} we derive that there exists a vertex $b\in B_1$ that is adjacent to all vertices from $A$, otherwise a set $S$ (i.e., a CD-set containing $\{u,y\}$, where $u \in N(x)$, and $x\notin S$) does not exist. Assuming that $S$ exists, Claim~\ref{l:linearOreder} implies that $D \cup \{b\}$ is a dominating set of $G$ for some $b\in B_1$. In addition, for one such $b\in B_1$ that dominates $A$ (such vertex may not be unique) we have $b\in S$. Therefore, as $\{u,y,b\}\subseteq S$, we infer $CH(\{u,y,b\})\subseteq CH(S)=S$. We conclude this case by noting that $R=\{u,y,b\}$ is a set of $G$ with at most four vertices (actually, with three vertices) such that $CH(R)\subseteq S$, and $CH(R)$ is a CD-set.

\smallskip

\noindent {\bf Case 1.2} $S$ is a CD-set containing $\{x,z\}$, where $z \in N(y)$, and $y \notin S$.

\smallskip

This case can be resolved in the same way as Case 1.1, by changing the roles of $x$ and $y$.

\smallskip

\noindent {\bf Case 1.3} $S$ is a CD-set containing $\{u,z\}$, where $u\in N(x), z \in N(y)$, and $x,y \notin S$.

\smallskip  

Let $D=CH(\{u,z\})$. Clearly, $D\subseteq S$.

\begin{claim}\label{l:AnotDomin1}
	Vertices in $V(G)\setminus (N(x) \cup N(y))$ are dominated by $D$.
\end{claim}
\begin{claimproof}
	Let $P$ be any shortest $u,z$-path of $G$. Since $D$ is a convex set containing $\{u,z\}$, $P$ is contained in $D$. As $(x,y)$ is a dominating pair, $\{x,y\}\cup V(P)$ consists of the vertices of an $x,y$-path and is thus a dominating set of $G$. Therefore $V(P)$ dominates all vertices of $G$ except perhaps some vertices from $N(x) \cup N(y)$. 
\end{claimproof}

Let $A$ be the set of vertices in $N(x)$ that are not dominated by $D$ and let $B\subseteq N(x)$ be the set of vertices not in $D$ but dominated by $D$, i.e.,\ $B=N(x)\cap (N[D]-D)$. Let $A'$ be the set of vertices in $N(y)$ that are not dominated by $D$ and let $B'\subseteq N(y)$ be the set of vertices not in $D$ but dominated by $D$, i.e.,\ $B'=N(y)\cap (N[D]-D)$. Note that $(A \cup B) \cap (A' \cup B')=\emptyset$, as $d(x,y) \geq 3$. Furthermore, let $L=V(G)-N[x]-N[y]-D$.

If $A=\emptyset=A'$, then $D=CH(\{u,z\})$ itself is a convex dominating set $S$ containing $\{u,z\}$, and we may take $R$ to be $\{u,z\}$ (where following the above notation $R$ is a set with at most 4 vertices such that $CH(R)\subseteq S$).

Without loss of generality we may assume one of $A,A'$ is not empty, and let $A\neq \emptyset$. We will establish some properties of sets $A,B, A',B'$.

\begin{claim}\label{l:Aisolated1}
	Let $A\neq \emptyset$, and $a$ is an arbitrary vertex from $A$. If $l \in L \cup D \cup N[y]$, then $al \notin E(G)$.
\end{claim}
\begin{claimproof}
	As $A$ is a set of vertices not dominated by $D$, $a$ has no neighbors in $D$. Since $d(x,y) \geq 3$, $ay \notin E(G)$. Suppose that $al\in E(G)$ for $l\in L\cup N(y)$. First let $l \in L \cup B'$, and let $d$ be a neighbor of $l$ in $D$, and $P$ a shortest $d,u$-path in $G$. Since $D$ is convex, $P \subseteq D$. Let $u'$ be the last neighbor of $l$ on $P$, and let $Q$ be the $u',u$-subpath of $P$. Since $au\notin E(G)$ and $xl\notin E(G)$, we derive that $x,a,l,Q,x$ is an induced cycle of length at least 4, a contradiction with $G$ being chordal. Finally, let $l \in A'$ and let $P$ be a shortest $z,u$-path, which is clearly contained in $D$. Note that $l$ has no neighbors on $P$, as $l \in A'$. Therefore the graph induced by $x,a,l,y,V(P)$ contains an induced cycle of length at least 4, a contradiction with $G$ being chordal.
\end{claimproof}

In the same way we can prove the following assertion:
\begin{claim}\label{l:A'isolated1}
	Let $A\neq \emptyset$, and $a$ is an arbitrary vertex from $A'$. If $l \in L \cup D \cup N[x]$, then $al \notin E(G)$.
\end{claim}

\begin{claim}\label{l:aNotInD1}
	If $a\in A \cup A'$, then $a\notin S$. 
\end{claim} 
\begin{claimproof}
We may assume without loss of generality that $a \in A$. Since $a$ is not dominated by $D$, $a$ is not adjacent to $u$. Therefore any convex set that contains $D \cup \{a\}$, also contains $x$, a contradiction with $x\notin S$. (The proof is analogous if $a\in A'$.)
\end{claimproof}

Since $A \neq \emptyset$ and $S$ is a dominating set, we have $N[A]\cap S\neq\emptyset$. From Claims~\ref{l:Aisolated1} and~\ref{l:aNotInD1} (and since $x\notin S$), we derive that $N[A]\cap S\subseteq B$. Let $B_1=N[A]\cap S$. (Clearly, $B_1\subseteq B$.) Following the same idea, let $B_1'=N[A'] \cap S.$ (If $A'=\emptyset$, then $B_1'=\emptyset$.)

\begin{claim}\label{l:Bclique1}
The set $B_1 \cup (D\cap N(x))$ is a clique.
\end{claim}
\begin{claimproof}
	Suppose that $b_1$ and $b_2$ are two nonadjacent vertices from $B_1 \cup (D\cap N(x))$. Then $b_1,x,b_2$ is a shortest $b_1,b_2$-path which implies that $x \in CH(D \cup B_1)(\subseteq S$), a contradiction.
\end{claimproof}

The following claim can be proved in the same way as Claim~\ref{l:Bclique1}.

\begin{claim}\label{l:B'clique1}
The set $B_1' \cup (D\cap N(y))$ is a clique.
\end{claim}

\begin{claim}\label{l:linearOreder1}
The sets from the family $\{N_{<A>}(b_i)\,:\, b_i \in B_1\}$ are linearly ordered with respect to inclusion.
\end{claim}
\begin{claimproof}
	Suppose that there exist $b_1,b_2 \in B_1$ such that $N_{<A>}(b_1) \nsubseteq N_{<A>}(b_2)$ and $N_{<A>}(b_2) \nsubseteq N_{<A>}(b_1)$. Therefore there exist $a_1,a_2\in A$ such that $a_1b_1,$ $a_2b_2 \in E(G)$, and $a_2b_1,a_1b_2 \notin E(G)$. It follows from Claim~\ref{l:Bclique1} that $b_1u, b_2u,$ $b_1b_2 \in E(G)$.
	Note that $a_1a_2 \notin E(G)$, as $G$ is chordal.  Furthermore, Claim~\ref{l:Aisolated1} implies that $a_1u,a_2u\notin E(G)$. Let $P=u,u_1,\ldots , u_k,z$ be a shortest $u,z$-path in $G$. As $D$ is convex, $P \subseteq D$. 
	
	Suppose first that $d(u,z) \geq 2$. Note that $u_2$ is not adjacent to any of $b_1,b_2$. Indeed, if $u_2b_1 \in E(G)$ ($u_2b_2\in E(G)$), then $b_1$ ($b_2$) lies on a shortest path between two vertices $u,u_2$ from $D$, which implies that $b_1$ ($b_2$) is in $D$, a contradiction. 
	If $u_1b_1,u_1b_2 \notin E(G)$, then vertices $a_1,a_2,b_1,b_2,u,u_1$ induce a graph $B_1$ from Figure~\ref{fig:Forbidden}, which implies that $G$ is not a chordal dominating pair graph, a contradiction. If $u_1$ is adjacent to one vertex from $\{b_1,b_2\}$, say $b_1$, then $a_1,a_2,b_1,b_2,u,u_1,u_2$ induce  a graph $B_2$, a contradiction. Finally, if $u_1$ is adjacent to both vertices $b_1$ and $b_2$, then $a_1,a_2,b_1,b_2,u_1,u_2$ induce a graph $B_1$, a contradiction.
	
	Finally, let $uz \in E(G)$. In this case $d(x,y) = 3$, hence $y$ is not adjacent to any of $a_1,a_2,b_1,b_2,u$. Again if $zb_1,zb_2 \notin E(G)$, then vertices $a_1,a_2,b_1,b_2,u,z$ induce a graph $B_1$ from Figure~\ref{fig:Forbidden}, which implies that $G$ is not a chordal dominating pair graph, a contradiction. If $z$ is adjacent to one vertex from $\{b_1,b_2\}$, say $b_1$, then $a_1,a_2,b_1,b_2,u,z,y$ induce  a graph $B_2$, a contradiction. Finally, if $u_1$ is adjacent to both vertices $b_1$ and $b_2$, then $a_1,a_2,b_1,b_2,z,y$ induce a graph $B_1$, a contradiction.
\end{claimproof}

In the same way the following assertion can be proved (note that if $A'=\emptyset$, the family in the assertion is also empty.)

\begin{claim}\label{l:linearOreder1'}
The sets from the family $\{N_{<A'>}(b_i)\,:\, b_i \in B_1'\}$ are linearly ordered with respect to inclusion.
\end{claim}

Note that if $A'=\emptyset$, then $B_1'=\emptyset$. We resolve this case in a similar (simplified) way, as the case when both $A,A'$ are non-empty, which we consider next. 

From Claims~\ref{l:linearOreder1} and \ref{l:linearOreder1'} we derive that there exist vertices $b\in B_1$, $b'\in B_1'$ such that $b$ is adjacent to all vertices from $A$ and $b'$ is adjacent to all vertices from $A'$, otherwise a set $S$ (i.e., a CD-set containing $\{u,z\}$, where $u \in N(x), z\in N(y)$, and $x,y\notin S$) does not exist. 
Assuming that $S$ exists, Claims~\ref{l:linearOreder1} and \ref{l:linearOreder1'} imply that $D \cup \{b,b'\}$ is a dominating set of $G$ for some $b\in B_1,b'\in B_1'$. In addition, for one such pair $(b,b')\in B_1 \times B_1'$ that dominates $A \cup A'$ (such pair may not be unique) we have $\{b,b'\}\in S$. Therefore, as $\{u,z,b,b'\}\subseteq S$, we infer $CH(\{u,z,b,b'\})\subseteq CH(S)=S$. We conclude this case by noting that $R=\{u,z,b,b'\}$ is a set of $G$ with four vertices such that $CH(R)\subseteq S$, and $CH(R)$ is a CD-set.

\medskip

\noindent {\bf Case 2} $d(x,y)=2$.

\smallskip

Let $U = N(x)\cap N(y)$, $X = N(x)\setminus U$, $ Y= N(y)\setminus U$, $W = V(G) \setminus (N[x]\cup N[y])$, $X_U = X \cap N(U)$, $Y_U = Y \cap N(U)$, $X' = X \setminus X_U$ and $Y' = Y \setminus Y_U$.
Since $d(x,y)=2$ it is clear that $U \neq \emptyset$. 

First we will prove some claims about the structure of the graph $G$.
\begin{claim}\label{l:Uclique}
	The subgraph of $G$ induced by $U$ is a complete graph.
\end{claim}
\begin{claimproof}
	Let $u,u'$ be arbitrary vertices from $U$. Since $G$ is chordal, the 4-cycle $x,u,y,u',x$ is not induced. Therefore $uu'\in E(G)$. 
\end{claimproof}

\begin{claim}\label{uDominates}
	If $u \in U$ and $w \in W$, then $uw \in E(G)$.
\end{claim}
\begin{claimproof}
	Since $(x, y)$ is a dominating pair of $G$, $\{x,u,y\}$ is a dominating set. Thus $u$ dominates all vertices from $W$. 
\end{claimproof}

\begin{claim}\label{l:2dominates}
	There exist at most two vertices in $U$ that dominate $N[U]$.
\end{claim}
\begin{claimproof}
	By Claim~\ref{l:Uclique}, $\left\langle  U \right\rangle$ is a complete graph. Hence, every $u \in U$ dominates $U$. If $|U|\leq 2$, it is clear, that there exist at most two vertices that dominate $N[U]$. Therefore let $|U| > 2$. Suppose that $N[U] \nsubseteq N(u)\cup N(v)$ for any pair $u,v \in U$. Hence, there exist $u_1,u_2,u_3 \in U$, $a_1,a_2,a_3 \in N[U]\setminus U$ such that $u_ia_i\in E(G)$ for any $i\in \{1,2,3\}$, and $a_iu_j \notin E(G)$ for any $i,j \in \{1,2,3\},i \neq j$. But then the graph induced by $a_1,a_2,a_3,u_1,u_2,u_3$ is either isomorphic to the forbidden induced subgraph $B_1$ from Figure~\ref{fig:Forbidden} or it contains an induced cycle of length at least 4, a contradiction.
\end{claimproof}

\begin{claim}\label{l:linearU}
	If $X' \neq \emptyset$ (resp. $Y' \neq \emptyset$), then the sets from the family $\{N(u)\cap Y\,:\, u \in U\}$ (resp. $\{N(u)\cap X\,:\, u \in U\}$) are linearly ordered with respect to inclusion.
\end{claim} 
\begin{claimproof}
	Let $X' \neq \emptyset$ and $x' \in X'$.
	By Claim~\ref{l:Uclique}, $\left\langle  U \right\rangle$ induces a complete graph. If $|U|= 1$, the claim holds. Therefore let $|U| \geq 2$. 
	Suppose that there exist $u_1, u_2 \in U$ such that $N(u_1)\cap Y\nsubseteq N(u_2)\cap Y$ and $N(u_2)\cap Y\nsubseteq N(u_1)\cap Y$. 
	Let $a_1 \in (N(u_1)\cap Y)\setminus N(u_2)$ and $a_2 \in (N(u_2)\cap Y)\setminus N(u_1)$.
	But then the graph induced by $a_1,a_2,x',u_1,u_2,x$ is either isomorphic to the forbidden induced subgraph $B_1$ from Figure~\ref{fig:Forbidden} or it contains an induced cycle of length at least 4, a contradiction.
	In the same way we can prove the claim for the family $\{N(u)\cap X\,:\, u \in U\}$, if $Y' \neq \emptyset$.
\end{claimproof}

\begin{claim}\label{l:neigh:X'}
	Let $x'$ be an arbitrary vertex from $X'$. If $x'a\in E(G)$, then $a \in X \cup \{x\}$.
\end{claim} 
\begin{claimproof}
	It follows from the definition of $X'$ that $a \notin U$. Suppose that $a \in V(G)\setminus(N[x] \cup Y')$. Then there exists $u \in U$ such that vertices $x',a,u,x,x'$ induce a 4-cycle, a contradiction. If $a \in Y'$, then the graph induced by $x,x',a,y,u,x$ is the 5-cycle, a contradiction. 
\end{claimproof}

In the same way one can prove the following claim.

\begin{claim}\label{l:neigh:Y'}
	Let $y'$ be an arbitrary vertex from $Y'$. If $y'a\in E(G)$, then $a \in Y \cup \{y\}$.
\end{claim}

\begin{claim}\label{l:XYtoU}
	Let $x' \in X$ and $y' \in Y$. If $x'y' \in E(G)$, then $x'u, y'u \in E(G)$ for all $u \in U$.
\end{claim} 
\begin{claimproof}
	Let $u \in U$. Since $G$ is chordal the 5-cycle $x,x',y',y,u,x$ is not induced. The only possible chords in this cycle are $x'u$ and $y'u$. Hence, $x'u, y'u \in E(G)$. 
\end{claimproof}

\begin{claim}\label{l:XYWtoU}
	Let $z \in X \cup Y$ and $w \in W$. If $zw \in E(G)$, then $uz \in E(G)$ for all $u \in U$.
\end{claim} 
\begin{claimproof}
	Let $u \in U$ and $z \in X$. Since $G$ is chordal the 4-cycle $x,z,w,u,x$ is not induced. The only possible chord in this cycle is $uz$. Hence, $uz \in E(G)$. In a similar way this can proved if $z \in Y$.
\end{claimproof}

We will first prove that if one of the conditions: $|U| \leq 2$, $X' = \emptyset$ or $Y' = \emptyset$ holds, then $\gc(G) \leq 4$, which by Remark~\ref{rem:cg4} implies the assertion of the theorem.
Let $|U|\leq 2$. Clearly, $\{x,y\}\cup U$ is a CD-set, since $x,y$ is a dominating pair and $U$ induces a complete graph (by Claim~\ref{l:Uclique}). Hence $\gc(G) \leq |\{x,y\}\cup U| \leq 4$.
Let $X' = \emptyset$. By Claim~\ref{l:2dominates} there exist $u,v \in U$ that dominate $N[U]$. Since $X = X_U \subseteq N(U) \subseteq N[U]$, $u,v$ dominate $X$. We claim that $\{u,v,y\}$ is a CD-set. Vertices in $W$ are dominated by $u,v$ and vertices in $N[y]$ by $y$. As $\left\langle  \{u,v,y\} \right\rangle$ is a clique, the set $\{u,v,y\}$ is a convex set. Hence, $\{u,v,y\}$ is a CD-set and $\gc(G) \leq 3$.
In the same way it can be proven for $Y' = \emptyset$, by changing the roles of $x$ and $y$.

Now we may restrict our attention to graphs $G$, where $|U| > 2$, $X' \neq \emptyset$ and $Y' \neq \emptyset$.
As in Case 1, let $S$ be a CD-set of $G$ with $\{x,y\}\nsubseteq S$.  Lemma~\ref{l:neighbours} implies that $S$ contains at least one vertex from $N[x]$ and at least one vertex from $N[y]$. We distinguish the following four cases:

\begin{enumerate}
	\item $S \cap U \neq \emptyset$ and $y \in S$, $x \notin S$.
	\item $S \cap U \neq \emptyset$ and $x \in S$, $y \notin S$.
	\item $S \cap U \neq \emptyset$ and $x,y \notin S$.
	\item $S \cap U = \emptyset$.
\end{enumerate}

To conclude the proof of the theorem for graphs where $d(x,y)=2$ it suffices to prove that in each of the above cases there exists a set of vertices $R$ with $|R|\le 4$ such that $CH(R)\subseteq S$ and $CH(R)$ is a CD-set of $G$.

\smallskip

\noindent {\bf Case 2.1} Let $S \cap U \neq \emptyset$ and $y \in S$, $x \notin S$.

\smallskip  

\begin{claim}\label{l:domX}
	Let $S$ be a CD-set of $G$, $x \notin S$ and let $u$ be a vertex in $S\cap U$ with the largest number of neighbors in $X$ among all vertices in $S\cap U$. Then there exists $z \in S \cap N(u) \cap X$ that dominates $X\setminus N(u)$.
\end{claim} 
\begin{claimproof}
	Let $S$ be a CD-set of $G$, where $x \notin S$ and let $u \in S\cap U$ be a vertex with the largest number of neighbors in $X$ among all vertices in $S\cap U$. First notice that $X' \subseteq X\setminus N(u)$. Hence, $X\setminus N(u) \neq \emptyset$.
	
	Since $S$ is a dominating set, for each $z' \in X\setminus N(u)$ there exists $s \in S$ for which $z's \in E(G)$.
	First, we will prove that such an $s$ is in $X$. Clearly, $s \neq x$ and $s \neq y$. If $s \in Y \cup W$, then by Claims~\ref{l:XYtoU} and \ref{l:XYWtoU}, $z'u \in E(G)$. Hence, $z' \in N(u)$, a contradiction. If $s \in U$, then $N(s)\cap X \nsubseteq N(u)\cap X$, a contradiction with Claim~\ref{l:linearU} and $u$ being a vertex in $S\cap U$ with the largest number of neighbors in $X$ among all vertices in $S\cap U$.
	Hence, $X\setminus N(u)$ is dominated, with respect to $S$, just by vertices in $S \cap X$.
	
	Next, we claim that $S \cap X \subseteq N(u) \cap X$. Suppose, that there exists $s \in S \cap (X \setminus N(u))$.
	Since $sx,xu \in E(G)$ and $su \notin E(G)$, $x$ is on a shortest $s,u$-path. Hence, $x \in S$, a contradiction.
	Therefore $X\setminus N(u)$ is dominated, with respect to $S$, just by vertices in $S \cap N(u) \cap X$.
	
	Suppose that there is no $z \in S \cap N(u) \cap X$ that dominates all vertices in $X \setminus N(u)$.
	Since $S$ is a dominating set, there exist $x_1, x_2 \in S \cap N(u) \cap X$ such that $(N(x_1)\cap X) \setminus (N(u) \cup N(x_2)) \neq \emptyset$ and $(N(x_2)\cap X) \setminus (N(u) \cup N(x_1)) \neq \emptyset$. 
	Let $a_1 \in (N(x_1)\cap X) \setminus (N(u) \cup N(x_2))$ and $a_2 \in (N(x_2)\cap X) \setminus (N(u) \cup N(x_1))$.
	Then the graph induced by $a_1,a_2,y,x_1,x_2,u$ is either isomorphic to the forbidden induced subgraph $B_1$ from Figure~\ref{fig:Forbidden} or it contains an induced cycle of length at least 4, a contradiction.
	Hence, there exists $z \in S \cap N(u) \cap X$ that dominates $X\setminus N(u)$.
\end{claimproof}

In a similar way one can prove the following claim.

\begin{claim}\label{l:domY}
	Let $S$ be a CD-set of $G$, $y \notin S$ and let $u$ be a vertex in $S\cap U$ with the largest number of neighbors in $Y$ among all vertices in $S\cap U$. Then there exists $z \in S \cap N(u) \cap Y$ that dominates $Y\setminus N(u)$.
\end{claim} 

Let $u \in S \cap U$ be a vertex in $S \cap U$ with the largest number of neighbors in $X$ among all vertices in $S\cap U$. By Claims~\ref{l:Uclique} and \ref{uDominates}, $u$ dominates $U\cup W \cup \{x,y\}$. Hence, only the vertices in $X \setminus N(u)$ are not dominated by $\{u,y\}$. By Claim~\ref{l:domX}, there exists $z \in S \cap N(u) \cap X$ that dominates $X \setminus N(u)$. Therefore, as $\{z, u, y\}\subseteq S$ is a dominating set, we infer $CH(\{z, u, y\})\subseteq CH(S)=S$. We conclude this case by noting that $R=\{z, u, y\}$ is a set of $G$ with three vertices such that $CH(R)\subseteq S$, and $CH(R)$ is a CD-set.
\smallskip  

\noindent {\bf Case 2.2} Let $S \cap U \neq \emptyset$ and $x \in S$, $y \notin S$.

\smallskip 
The desired assertion can be proven in a similar way as Case 2.1 by changing the roles of $x$ and $y$, and by using Claim~\ref{l:domY}.
\smallskip  

\noindent {\bf Case 2.3} Let $S \cap U \neq \emptyset$ and $x,y \notin S$.

\smallskip
Let $u \in S \cap U$ be a vertex in $S \cap U$ with the largest number of neighbors in $X$ among all vertices in $S\cap U$ and let $v \in S \cap U$ be a vertex in $S \cap U$ with the largest number of neighbors in $Y$ among all vertices in $S\cap U$ (note, that $u$ and $v$ can coincide). By Claims~\ref{l:Uclique} and \ref{uDominates}, $\{u,v\}$ dominates $U\cup W \cup \{x,y\}$. Hence, only vertices in $X \setminus N(u)$ and $Y \setminus N(v)$ are not dominated by $\{u,v\}$. By Claim~\ref{l:domX}, there exists $z_X \in S \cap N(u) \cap X$ that dominates $X \setminus N(u)$ and by Claim~\ref{l:domY}, there exists $z_Y \in S \cap N(v) \cap Y$ that dominates $Y \setminus N(v)$. Therefore, as $\{z_X, u, y, z_Y\}\subseteq S$ is a dominating set, we infer $CH(\{z_X, u, y, z_Y\})\subseteq CH(S)=S$. We conclude this case by noting that $R=\{z_X, u, y, z_Y\}$ is a set of $G$ with at most four vertices (it can happen that $u=v$) such that $CH(R)\subseteq S$, and $CH(R)$ is a CD-set.

\medskip  

\noindent {\bf Case 2.4} Let $S \cap U = \emptyset$.

\smallskip

\begin{claim}\label{l:domXX}
	Let $S$ be a CD-set of $G$, $x,y \notin S$, $S\cap U = \emptyset$, and let there exist $x' \in X \cap S, y' \in Y \cap S$ such that $x'y' \in E(G)$. If $X \setminus N(x') \neq \emptyset$, then there exists $z \in S \cap N(x') \cap X$ that dominates $X\setminus N(x')$.
\end{claim} 
\begin{claimproof}
	Let $S$ be a CD-set of $G$ as described above and $X \setminus N(x') \neq \emptyset$. By Claim~\ref{l:XYtoU}, $x'u,y'u \in E(G)$ for all $u \in U$.
	Since $S$ is a dominating set, for each $z' \in X\setminus N(x')$ there exists $s \in S$ for which $z's \in E(G)$.
	First, we will prove that such an $s$ is in $X$. Clearly, $s \notin U \cup \{x, y\}$. If $s \in Y \cup W$, then by Claims~\ref{uDominates}, \ref{l:XYtoU} and \ref{l:XYWtoU}, $z'u,su \in E(G)$ for all $u \in U$. Since $S\cap U = \emptyset$, $x's \in E(G)$.
	But then $x,x',s,z',x$ induces a 4-cycle, a contradiction.
	Hence, $X\setminus N(x')$ is dominated, with respect to $S$, just by vertices in $S \cap X$.
	
	Next, we will prove that $S \cap X \subseteq N(x') \cap X$. Suppose, that there exists $s \in S \cap (X \setminus N(x'))$.
	Since $sx,xx' \in E(G)$ and $sx' \notin E(G)$, $x$ is on a shortest $s,x'$-path. Hence, $x \in S$, a contradiction.
	Therefore $X\setminus N(x')$ is dominated, with respect to $S$, just by vertices in $S \cap N(x') \cap X$.
	
	Suppose that there is no $z \in X\cap N(x')\cap S$ that dominates all vertices in $X \setminus N(x')$.
	Since $S$ is a dominating set, there exist $x_1, x_2 \in S \cap N(x') \cap X$ such that $(N(x_1)\cap X) \setminus (N(x') \cup N(x_2)) \neq \emptyset$ and $(N(x_2)\cap X) \setminus (N(x') \cup N(x_1)) \neq \emptyset$. 
	Let $a_1 \in (N(x_1)\cap X) \setminus (N(x') \cup N(x_2))$ and $a_2 \in (N(x_2)\cap X) \setminus (N(x') \cup N(x_1))$.
	Since $x \notin S$, $x_1x_2 \in E(G)$.
	Notice, that $a_1, a_2 \in X\setminus N(x')$. Hence, $a_1y',a_2y' \notin E(G)$ (vertices in $X\setminus N(x')$ are dominated, with respect to $S$, just by vertices in $S \cap X$).
	Now, we distinguish three cases: 
	\begin{enumerate}
		\item $x_1y', x_2y' \notin E(G)$
		\item $x_1y' \in E(G)$ and $x_2y' \notin E(G)$, or $x_2y' \in E(G)$ and $x_1y' \notin E(G)$.
		\item $x_1y', x_2y' \in E(G)$
	\end{enumerate}
	If $x_1y', x_2y' \notin E(G)$, then the graph induced by $a_1,a_2,y',x_1,x_2,x'$ is either isomorphic to the forbidden induced subgraph $B_1$ from Figure~\ref{fig:Forbidden} or it contains an induced cycle of length at least 4, a contradiction.
	
	If $x_1y' \in E(G)$ and $x_2y' \notin E(G)$, then the graph induced by $a_1,a_2,y,x_1,x_2,$ $y',x'$ is either isomorphic to the forbidden induced subgraph $B_2$ from Figure~\ref{fig:Forbidden} or it contains an induced cycle of length at least 4, a contradiction. The case when $x_2y' \in E(G)$ and $x_1y' \notin E(G)$, yields a contradiction in a similar way (by changing the roles of $x_1$ and $x_2$).
	
	Finally, suppose $x_1y', x_2y' \in E(G)$. In this last case the graph induced by $a_1,a_2,y,x_1,x_2,y'$ is either isomorphic to the forbidden induced subgraph $B_1$ from Figure~\ref{fig:Forbidden} or it contains an induced cycle of length at least 4, a contradiction.
	
	Hence, there exists $z \in S \cap N(x') \cap X$ that dominates $X\setminus N(x')$.
\end{claimproof}

In the same way one can prove the following claim.

\begin{claim}\label{l:domYY}
	Let $S$ be a CD-set of $G$, $x,y \notin S$, $S\cap U = \emptyset$, and there exists $x' \in X \cap S, y' \in Y \cap S$ such that $x'y' \in E(G)$. If $Y \setminus N(y') \neq \emptyset$, then there exists $z \in S \cap N(y') \cap Y$ that dominates $Y\setminus N(y')$.
\end{claim} 

First, we will prove that $x, y \notin S$. Suppose that $y \in S$. Hence, $x \notin S$. Since $S$ is a dominating set and $x$ has to be dominated, there exists $x' \in X \cap S$. We claim that $d(x',y) = 2$. Indeed, if $d(x',y) > 2$, then $x$ is on a shortest $x',y$-path, a contradiction with $x \notin S$. Therefore there exists $y' \in Y$ such that $x'y' \in E(G)$ (otherwise $S \cap U \neq \emptyset$, a contradiction). By Claim~\ref{l:XYtoU}, $x'u, y'u \in E(G)$ for all $u \in U$. Hence, $u$ is on a shortest $x',y$-path, a contradiction with $S \cap U = \emptyset$. 
The same way we can prove that $x \notin S$.

Next, we will prove that there exist $x' \in X \cap S, y' \in Y \cap S$ such that $x'y' \in E(G)$. Suppose that $x'y' \notin E(G)$ for any $x' \in X \cap S, y' \in Y \cap S$. Let $x' \in X \cap S$ and $y' \in Y \cap S$ (they exist by Lemma~\ref{l:neighbours}) have the shortest distance among all such pairs. Let $P$ be a shortest $x',y'$-path. Clearly, $P \subseteq S$ and $P \subseteq X \cup Y \cup W$. Hence, $P = x', w_1, w_2,\ldots, w_k, y'$ where $k\geq 1$ and $w_i \in W$ for all $i \in {1,\ldots,k}$. Let $u \in U$. Since $G$ is chordal the cycle $x, x', w_1, u, x$ is not induced. The only possible chord in this cycle is $ux'$, hence, $ux' \in E(G)$. The same way we can prove that $uy' \in E(G)$, as $uy'$ is necessarily a chord in the cycle $y, y', w_k, u, y$. Therefore $u$ is on a shortest $x',y'$-path, a contradiction with $S \cap U = \emptyset$.

Let $x' \in X \cap S, y' \in Y \cap S$ such that $x'y' \in E(G)$.
By Claim~\ref{l:XYtoU}, $x'u, y'u \in E(G)$ for all $u \in U$. 
Hence, $U$ is dominated by $\{x', y'\}$. Since $x,y$ is a dominating pair, the path $x,x',y',y$ dominates all vertices of $G$. Therefore, $\{x', y'\}$ dominates $W$. This implies that just vertices from $X$ and $Y$ need not be dominated by $\{x', y'\}$.

If $X$ and $Y$ are dominated by $\{x', y'\}$, then $\{x', y'\}$ is a CD-set and $\gc(G) \leq 2$, which by Remark~\ref{rem:cg4} implies the assertion of the theorem.

Suppose that $X$ is not dominated by $\{x', y'\}$ and $Y$ is dominated by $\{x', y'\}$.
By Claim~\ref{l:domXX} there exists $z \in X\cap N(x')\cap S$ that dominates all vertices in $X \setminus N(x')$.
Therefore, as $\{z, x', y'\}\subseteq S$ is a dominating set, we infer $CH(\{z, x', y'\})\subseteq CH(S)=S$. We conclude this case by noting that $R=\{z, x', y'\}$ is a set of $G$ with three vertices such that $CH(R)\subseteq S$, and $CH(R)$ is a CD-set.

The same way it can be proven, that if $X$ is dominated by $\{x', y'\}$ and $Y$ is not, then by Claim~\ref{l:domYY} there exists $z \in Y\cap N(y')\cap S$, that dominates all vertices in $Y \setminus N(y')$. In this case $R=\{x', y', z\}$.
Finally, if $X$ and $Y$ are not dominated by $\{x', y'\}$, then by Claim~\ref{l:domXX} there exist $z_X \in X\cap N(x')\cap S$ that dominates all vertices in $X \setminus N(x')$, and by Claim~\ref{l:domYY} there exists $z_Y \in Y\cap N(y')\cap S$, that dominates all vertices in $Y \setminus N(y')$. In this case $R=\{z_X, x', y', z_Y\}$.
\end{proof}

\begin{corollary}
	Let $G$ be a chordal dominating pair graph. Then a minimum convex dominating set can be found in polynomial time.
\end{corollary}
\begin{proof}
	We present an algorithm that finds a minimum convex dominating set of a chordal dominating pair graph $G$.
	\begin{enumerate}
		\item Compute convex hulls of all $R\subseteq V(G)$ with $|R|\le 4$.
		\item For each convex hull check whether it is a dominating set.
		\item From all convex hulls that are dominating sets choose the smallest one.
	\end{enumerate} 
	Since Theorem~\ref{thm:main} implies that there exists a set $R\subseteq V(G)$ with $|R|\le 4$ such that $CH(R)$ is a minimum convex dominating set, the above algorithm finds minimum convex dominating set of a chordal dominating pair graph.
	
	The complexity of the algorithm is polynomial. Indeed, for $n=|V(G)|$ there are $O(n^4)$ subsets $R\subseteq V(G)$ with $|R|\le 4$; the convex hull of a subset of $V(G)$ can be computed in polynomial time; and checking if a set is a dominating set can be also realized in polynomial time. 
	
\end{proof}

\section{Isometric domination of weak dominating pair graphs}
\label{sec:iso}

In this section we give the polynomial time algorithm to determine the isometric domination number of a weak dominating pair graph. 

The following lemma is the first step towards the proof of the main result in this section, and follows from definitions of the involved parameters.

\begin{lemma}\label{l:ista}
	Let $G$ be a weak dominating pair graph and let $(x,y)$ be a dominating pair. Then any isometric dominating set $S$ of $G$ contains at least one vertex from $N[x]$ and at least one vertex from $N[y]$. In addition, if $x$ (respectively, $y$) belongs to $S$, then $S$ contains a vertex in $N(x)$ (respectively, $N(y)$).
\end{lemma}

\begin{lemma}\label{3cases}
Let $G$ be a weak dominating pair graph and let $(x,y)$ be a dominating pair. Then $$d_G(x,y)-1 \leq \gamma_{iso}(G) \leq d_G(x,y)+1.$$
\end{lemma}
\begin{proof}
If $S$ is a $\gi$-set in a dominating pair graph $G$, then, by Lemma~\ref{l:ista}, $N[x]\cap S\neq \emptyset$ and $N[y]\cap S\neq \emptyset$. Let $a\in N(x)\cap S$ and $b\in N(y)\cap S$. Since $S$ is isometric, a shortest $a,b$-path $P$ is also in $S$, and it is clear that $|V(P)|=d_G(a,b)+1 \ge (d_G(x,y)-2)+1 = d_G(x,y)-1$. We derive that $|S|\ge |V(P)|\ge d_G(x,y)-1$.

To prove the right-hand inequality, note that a shortest $x,y$-path $Q$ is an isometric dominating set of $G$, and so $\gamma_{iso}(G) \leq d_G(x,y)+1$.

\end{proof}

\begin{theorem} The isometric domination number of a weak dominating pair graph in which a dominating pair is given can be computed in polynomial time.
\end{theorem}
\begin{proof}
Let $G$ be a graph with dominating pair $(x,y)$. If $\gi(G) \leq 4$, $\gi$-set can be found by exhaustively checking all $k$-tuples of vertices, for $k\le 4$, for being ID-sets or not, which can be done in polynomial time. Therefore we may assume that $\gi(G) > 4$, and $d(x,y) > 3$, for otherwise a shortest $x,y$-path is an ID-set in $G$ of length at most 4.

Let $x,x_1,\ldots , x_k,y$ be a shortest $x,y$-path in $G$. Note first that $d_G(a,b) \geq d_G(x_1,x_k)=d_G(x,y)-2$ for any $a \in N(x)$, $b \in N(y)$. 
Let $S$ be minimum ID-set of $G$. By Lemma~\ref{l:ista}, there exist vertices $a\in N(x)\cap S$ and $b\in N(y)\cap S$. Since $S$ is isometric, a shortest $a,b$-path $P$ is also in $S$, and $|V(P)| = d_G(a,b)+1 \ge d_G(x,y)-1$. Now, by Lemma~\ref{3cases}, we have three possibilities. 

Firstly, if $|S|=d_G(x,y)-1$, then $S=V(P)$. 

Secondly, if $|S|=d_G(x,y)$, then two cases occur. If $S=V(P)$, then $d_G(a,b)=d_G(x,y)-1$. The second case is that $V(P)\subsetneq S$. This readily implies that $d_G(a,b)=d_G(x,y)-2$ ($a,b$ lies on a shortest $x,y$-path), and $|S\setminus V(P)|=1$. In this case $V(P)$ dominates all vertices of $G$ except some vertices from $N(x) \cup N(y)$. If also $N(x)$ is dominated by $V(P)$, then $V(P) \cup \{y\}$ is an ID-set. If all vertices from $N(y)$ are dominated by $V(P)$, then $V(P) \cup \{x\}$ is an ID-set. Finally suppose that there exist $\emptyset \neq X' \subseteq N(x)$ that is not dominated by $V(P)$ and $\emptyset \neq Y' \subseteq N(y)$ that is not dominated by V(P). Since $|S \setminus V(P)|=1$, there exists $u \in S \setminus V(P)$ that dominates $X' \cup Y'$. As $S$ is isometric, $u$ is adjacent to at least one vertex $v \in V(P)$. Let $x' \in X'$, $y' \in Y'$. Since $a,b$ lies on a shortest $x,y$-path and $d(x',y')=2$ we get $d(a,b) \leq 2$. Therefore $|S|=d_G(x,y) \leq 4$, a contradiction (with the assumption $\gi(G) > 4$).

Finally, if $|S|=d_G(x,y)+1$, then a shortest $x,y$-path has $d_G(x,y)+1$ vertices, and is an ID-set. 

\medskip 

The above arguments imply that the minimum isometric dominating set can be found in the following way.

Resolving the possibility that $\gi(G) \leq 4$, we start the algorithm by performing exhaustive check of all $k$-tuples of vertices, for $k\le 4$, for being ID-sets or not.
Next, let $s=d_G(x,y)-2$, and we may assume that $s\ge 2$. For any $a\in N(x)$,$b\in N(y)$ compute $d_G(a,b)$. For all pairs $a\in N(x)$,$b\in N(y)$ with $d_G(a,b)=s$ compute all shortest paths between $a$ and $b$, and check whether any of these paths dominates $G$. If there is such a path between a pair $a\in N(x)$,$b\in N(y)$ with $d_G(a,b)=s$ that dominates $G$, then $\gi(G)=s+1=d_G(x,y)-1$. Otherwise, 
if there is some path $P$ between a pair $a\in N(x)$,$b\in N(y)$ with $d_G(a,b)=s$ such that only some vertices in $N(x)$ (respectively, $N(y)$) are not dominated by $V(P)$, then $V(P)\cup\{x\}$ (respectively, $V(P)\cup\{y\}$) is an ID-set of $G$, and $\gi(G)=s+2=d_G(x,y)$. 
Otherwise for all pairs $a\in N(x)$,$b\in N(y)$ with $d_G(a,b)=s+1$ compute all shortest paths between $a$ and $b$, and check whether any of these paths dominates $G$. If there is such a path $P$, then $V(P)$ is an ID-set of $G$, and $\gi(G)=s+2=d_G(x,y)$. Otherwise $\gi(G)=s+3=d_G(x,y)+1$ and a shortest $x,y$-path is an ID-set.

\medskip 
It remains to prove that $\gi$-set can be found in polynomial time.

To compute the distances between pairs of vertices in $N(x)\times N(y)$ is clearly polynomial. For a given pair $(a,b)$ computing all shortest paths is also polynomial. Note that checking whether some of these paths dominates $G$, one can restrict only to vertices $a$ and $b$ and their neighbors in the set of vertices on shortest $a,b$-paths. This can again be done in polynomial time.
\end{proof}

\section*{Acknowledgements}
The authors acknowledge the financial support from the Slovenian Research Agency (research core funding No.\ P1-0297); B.B.\ and T.G.\ also acknowledge the support from the Slovenian Research Agency by the project grant N1-0043.


\end{document}